\documentclass[11pt]{amsart}
\pdfoutput=1

\usepackage[dvipsnames, table, xcdraw]{xcolor}
\usepackage{upgreek}
\usepackage{amssymb,amsmath,mathtools,amsfonts,amsthm,thmtools}
\usepackage[mathscr]{eucal}
\usepackage{fullpage}
\usepackage{color}
\usepackage{caption}
\usepackage{subcaption}
\usepackage{listings}
\usepackage{enumitem}
\setlist[itemize]{leftmargin=30pt, itemsep=2pt}
\setlist[enumerate]{leftmargin=30pt, itemsep=2pt}
\usepackage{soul}
\usepackage{setspace} 
\usepackage[numbers]{natbib}
\usepackage{indentfirst} 
 
\usepackage[colorlinks,pagebackref]{hyperref}

\definecolor{mylinkcolor}{RGB}{0,0,255}
\definecolor{mycitecolor}{RGB}{169,169,169}
\definecolor{myurlcolor}{RGB}{255,20,147}
\definecolor{mybiburlcolor}{RGB}{80,80,80}
\hypersetup{
  colorlinks=true,
  urlcolor=myurlcolor,
  citecolor=mycitecolor,
  linkcolor=mylinkcolor,
  linktoc=page,
  breaklinks=true
}
\usepackage{tikz-cd,tikz} 
\usepackage{cleveref} 

\usepackage{algorithmic}
\usepackage[ruled]{algorithm}
\makeatletter 
 
\@addtoreset{algorithm}{section} 
\makeatother

\numberwithin{equation}{section}

\newtheorem{theorem}[algorithm]{Theorem}
\newtheorem{lemma}[algorithm]{Lemma}
\newtheorem{coro}[algorithm]{Corollary}
\newtheorem{conjecture}[algorithm]{Conjecture}

\newtheorem{proposition}[algorithm]{Proposition}

\theoremstyle{definition} 
\newtheorem{defn}[algorithm]{Definition}
\newtheorem{question}[algorithm]{Question}
\newtheorem{remark}[algorithm]{Remark}

\newtheorem{example}[algorithm]{Example}

\crefname{equation}{equation}{equations}
\Crefname{equation}{Equation}{Equations}
\crefname{conjecture}{conjecture}{conjectures}
\Crefname{conjecture}{Conjecture}{Conjectures}

\usepackage{colonequals}

\newcommand{\ZZ}{\mathbb{Z}}
\newcommand{\QQ}{\mathbb{Q}} 
\newcommand{\CC}{\mathbb{C}} 
\newcommand{\PP}{\mathbb{P}} 
\renewcommand{\AA}{\mathbb{A}} 
\newcommand{\cE}{\mathcal{E}}

\newcommand{\cO}{\mathcal{O}}

\DeclareMathOperator{\Aut}{Aut}
\DeclareMathOperator{\codim}{codim} 
\DeclareMathOperator{\rank}{rk} 
\DeclareMathOperator{\Spec}{Spec}
 
\DeclareMathOperator{\Ext}{Ext} 
\DeclareMathOperator{\Sym}{Sym} 
\DeclareMathOperator{\Tsch}{Tsch} 
 
\DeclareMathOperator{\Bun}{Bun}

\DeclareMathOperator{\End}{End}
\DeclareMathOperator{\Pic}{Pic}

\renewcommand{\subset}{\subseteq}

\newcommand{\relSpec}{\Spec}
\newcommand{\rest}{(-e_5,-e_4,-e_3)}
\newcommand{\fac}{\mathsf{fac}}
\newcommand{\bad}{\mathsf{bad}}
\DeclareFontShape{OT1}{cmr}{m}{scit}{<->ssub*cmr/m/sc}{}

\title[Tschirnhausen bundles of sextic covers of $\PP^1$]{Tschirnhausen bundles of sextic covers of $\PP^1$}
\date{06 April 2026}

\author{Sam Frengley} 
\address{Inria and École polytechnique, Institut Polytechnique de Paris, Palaiseau, France}
\email{samuel.frengley@inria.fr}
\urladdr{\url{https://samfrengley.github.io/}}

\author{Sameera Vemulapalli}
\address{Department of Mathematics,
  Harvard University}
\email{vemulapalli@math.harvard.edu}
\urladdr{\url{https://web.math.princeton.edu/~sameerav/}}

\allowdisplaybreaks

\begin{document}

\begin{abstract}
  A degree $d$ genus $g$ cover of the complex projective line by a smooth irreducible curve $C$ yields a vector bundle on the projective line by pushforward of the structure sheaf. We classify the bundles that arise this way when $d = 6$. Interestingly, our methods show that all constraints on the pushforward are ``explained'' by multiplication in an algebra. Finally, we show that all possible pushforwards are realized by covers with a nontrivial proper subcover.
\end{abstract}

\maketitle

\begin{figure}[ht]
  \centering
  \includegraphics[scale=0.25]{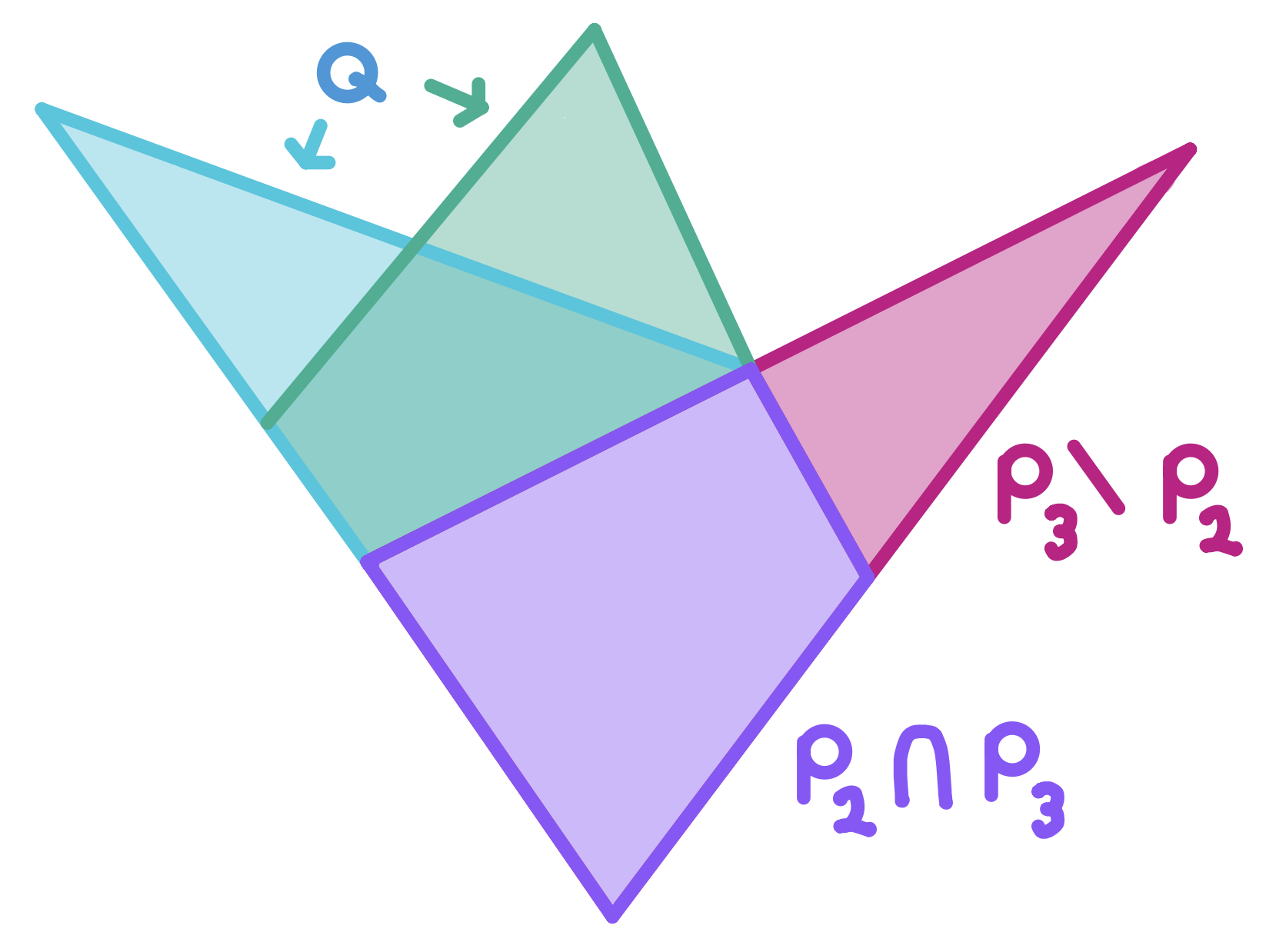}
  \caption{
    A cartoon drawing of the geography of degree $6$ covers of $\PP^1$. Given a sextic cover $\pi \colon C \rightarrow \PP^1$, its scrollar invariants $(e_1,\dots,e_5)$ yield an integer point of this region. We prove that every integer point of this region arises as the scrollar invariants of a sextic curve. The covers in the pink region $\mathcal{P}_3 \setminus \mathcal{P}_2$ necessarily factor through cubic subcovers. The covers in the blue and green regions $\mathcal{Q}$ necessarily factor through quadratic subcovers. The covers which do not factor through a nontrivial proper subcover are contained in the purple region $\mathcal{P}_2 \cap \mathcal{P}_3$.
  }
  \label{f:shapes}
\end{figure}

\setcounter{tocdepth}{1}
\tableofcontents

\section{Introduction}
Let $C$ be a smooth irreducible projective curve of genus $g$ over $\CC$ equipped with a finite morphism $\pi \colon C \to \PP^1$ of degree $d$. 
The restriction of functions gives an inclusion $\cO_{\PP^1} \rightarrow \pi_* \cO_C$.  
Define the \emph{Tschirnhausen bundle of $\pi$} to be $\mathcal{E}_{\pi} \coloneqq (\pi_* \cO_C /\cO_{\PP^1})^{\vee}$. By the Birkhoff-Grothendieck theorem, vector bundles on $\PP^1$ split completely, so we may write (uniquely up to non-unique isomorphism)
\[
\cE_\pi \cong \cO_{\PP^1}(e_1)\oplus \cdots \oplus \cO_{\PP^1}(e_{d-1})
\]
with $e_1 \leq e_2 \leq \dots \leq e_{d-1}$. Because $h^0(\PP^1, \pi_* \cO_C) = h^0(C,\cO_C) = 1$, we have $e_1 \geq 1$. The integers $\vec{e}_\pi = (e_1, \dots, e_{d-1})$ are called the \emph{scrollar invariants} of $\pi$, and encode a good deal of information about the curve $C$. For the rest of this article, assume $1 \leq e_1 \leq \dots \leq e_{d-1}$.

\begin{question}[Tschirnhausen realization problem]
  \label{q:main}
  Which $(d-1)$-tuples $\vec{e} \in \ZZ^{d-1}$ arise as the scrollar invariants of a smooth irreducible degree $d$ cover $\pi \colon C \rightarrow \PP^1$?
\end{question}

For $2 \leq d \leq 5$, the answer to the Tschirnhausen realization problem is known completely (combining \cite{maroni}, \cite{kato-ohbuchi}, \cite[Theorem~3.1]{VV}, and \cite[Theorem~1.4]{FV}) using the parametrizations of low degree covers \cite{CE,casnati}. Moreover, when $d$ is prime there is a conjectural answer to the Tschirnhausen realization problem \cite[Conjecture 1.3]{VV}; no such conjecture exists for composite $d > 5$. In \Cref{thm:main-thm} we give a complete answer to the Tschirnhausen realization problem in the case $d = 6$. 

Define $\mathcal{P}_2, \mathcal{P}_3 \subset \ZZ^5$ by the following conditions
\[
  \mathcal{P}_2 \colonequals \left\{
    ( e_1,\dots,e_5) \in \ZZ^5 \quad \Bigg| \quad  
    \begin{aligned} 
      & 1 \leq e_1 \leq \dots \leq e_5, \; e_5 \leq e_1 + e_4, \\ 
      & e_5 \leq e_2 + e_3, \; e_3 \leq e_1 + e_2, \; e_4 \leq 2e_2
    \end{aligned}
  \right\}
\]

\[
  \mathcal{P}_3 \colonequals \left\{
    (e_1,\dots,e_5) \in \ZZ^5 \quad \Bigg| \quad 
    \begin{aligned}
      & 1 \leq e_1 \leq \dots \leq e_5, \; e_5 \leq e_1 + e_4, \; \\
      & e_5 \leq e_2 + e_3,\; e_2 \leq 2e_1, \; e_4 \leq e_1 + e_3 
    \end{aligned}
  \right\}.
\]

\begin{defn}
  \label{def:admissible-partition}
  Let $\vec{e} = (e_1,\dots,e_5) \in \mathcal{P}_2$. We say a partition $\{i,j\} \sqcup \{k,\ell\} = \{2,3,4,5\}$ is \emph{admissible for $\vec{e}$} if the following conditions hold: 
  \begin{enumerate}[label=(\arabic*)]
  \item \label{cond:1} $e_i + e_j \leq e_k + e_\ell$,
  \item \label{cond:2} $e_i \leq e_j \leq e_i + e_1$,
  \item \label{cond:3} $e_k \leq e_{\ell} \leq e_k + e_1$, and
  \item \label{cond:4} $e_k + e_\ell + e_1 \leq 2e_i + 2e_j$.
  \end{enumerate}
  Let $\mathcal{Q}$ be the set of points $(e_1,\dots,e_5) \in \mathcal{P}_2 \setminus \mathcal{P}_3$ admitting an admissible partition.
\end{defn}

\begin{theorem}
  \label{thm:main-thm}
  A quintuple $\vec{e} \in \ZZ^5$ arises as the scrollar invariants of a smooth irreducible sextic cover if and only if $\vec{e} \in \mathcal{Q} \cup \mathcal{P}_3$.
\end{theorem}

\subsection{Constraints on scrollar invariants}
The scrollar invariants of a degree $d$ cover $\pi \colon C \to \PP^1$ must satisfy constraints arising from multiplication. More precisely, $\pi_* \cO_C$ is an $\cO_{\PP^1}$-algebra and this forces $\vec{e}_{\pi}$ to satisfy some linear inequalities \cite[Section~1.3]{Vemulapalli:bounds} (e.g., see \Cref{ex:multiplication-constraints} for a proof that $e_2 \leq 2e_1$ when $d = 3$). One significant feature is that the linear inequalities imposed on $\vec{e}_\pi$ depend on the degrees of the subcovers of $\pi$. In the case $d \leq 5$, the necessary constraints arising from multiplication in \cite[Section~1.3]{Vemulapalli:bounds} are known to be sufficient~\cite{maroni, kato-ohbuchi, VV, FV}.

Moreover, when $\pi$ is \emph{primitive} (that is, if $\pi$ does not factor through a nontrivial proper subcover) it is known from \cite[Section~1.3]{Vemulapalli:bounds} that $e_{i + j} \leq e_i + e_j$ for all $i$ and $j$. A conjecture of the second author~\cite[Conjecture 1.3]{VV} asserts that this is a \emph{sufficient} condition, i.e., every $\vec{e}$ satisfying $e_{i + j} \leq e_i + e_j$ arises as the scrollar invariants of a smooth primitive cover. This conjecture is known for each $d \leq 5$ (combining \cite{maroni, VV, FV}).

In the imprimitive case, the constraints in \cite[Section~1.3]{Vemulapalli:bounds} are significantly more complicated. Regardless, they are worked out explicitly for all $d < 18$ \cite[Theorem~1.26]{Vemulapalli:bounds}. One might wonder if these are \emph{sufficient} conditions, but this turns out to be false (as we show in the case $d = 6$). There is a good reason for this insufficiency: if $\pi \colon C \rightarrow \PP^1$ factors through a subcover as $\tau \circ \gamma$, then $\gamma_* \cO_C$ is an algebra over the intermediate cover and this forces the slopes of $\gamma_* \cO_C$ to satisfy some linear inequalities. This in turn imposes constraints on $\vec{e}_{\pi}$. 

Indeed, in the case $d = 6$, if $\pi$ is a sextic cover factoring as $\tau \circ \gamma$ through a degree $2$ cover $\tau \colon X \to \PP^1$, then $\vec{e}_\pi$ satisfies additional linear inequalities coming from the multiplicative structure of $\gamma_* \cO_C$ as a cubic $\cO_X$-algebra. This constraint is truly new in the sense that it is \emph{not} implied by those in \cite[Section~1.3]{Vemulapalli:bounds}. We prove via explicit construction that this new constraint, combined with the constraints of \cite[Section~1.3]{Vemulapalli:bounds}, are necessary and sufficient conditions for a quintuple $\vec{e}$ to arise as the scrollar invariants of a smooth sextic cover.

\subsection{Main ideas of the proof of \Cref{thm:main-thm}}
If $\pi \colon C \rightarrow \PP^1$ is a degree $6$ cover \cite[Section~1.3]{Vemulapalli:bounds} gives an explicit set of linear inequalities on the scrollar invariants $\vec{e}_\pi$, depending on how $\pi$ factors. It turns out that in this case (when $d = 6$) there are quintuples $\vec{e}$ which satisfy the constraints of \cite[Section~1.3]{Vemulapalli:bounds} but which \emph{do not} arise as the scrollar invariants of a sextic curve. For example if $e_2 > 2e_1$, then one can show (\Cref{lem:factoring}) that $C$ is a triple cover of a hyperelliptic curve $\tau \colon X \rightarrow \PP^1$, e.g., there is a factorization
\[
    \begin{tikzcd}
C \arrow[swap, rd, "\pi"] \arrow[r, "\gamma"] & X \arrow[d, "\tau"]\\
& \PP^1
\end{tikzcd}
\]
and thus $\gamma_* \cO_C$ is a cubic $\cO_X$-algebra. The multiplication structure of $\gamma_* \cO_C$ places nontrivial constraints on the slopes of the vector bundle $\gamma_* \cO_C$, which in turn places constraints on $\vec{e}_{\pi}$. These constraints, which do not follow from \cite[Section~1.3]{Vemulapalli:bounds}, play a central role in this article. 

After proving these constraints (which we do in \Cref{thm:containment}), it remains to show that all $\vec{e}$ satisfying these constraints actually arise as the scrollar invariants of a sextic cover. Remarkably it turns out that such a quintuple $\vec{e}$ can always be realized as the scrollar invariants of a cover which factors (that is, a double cover of a triple cover or a triple cover of a double cover).

\subsection{Future directions}
We now pose several open questions for interested reader which may be tractable via the methods of this paper.

\subsubsection{Tschirnhausen bundles of primitive degree $6$ covers}
A conjecture of the second author~\cite[Conjecture 1.3]{VV} asserts that every $\vec{e}$ satisfying $e_{i + j} \leq e_i + e_j$ arises as the scrollar invariants of a smooth primitive cover. The case $d = 6$ is the smallest degree in which this conjecture is not known. In \cite[Theorem~1.4]{VV} it is shown that if $\vec{e}_\pi$ is the scrollar invariants of a primitive cover, then $e_{i + j} \leq e_i + e_j$ for all $i$ and $j$. Thus the difficulty of proving \cite[Conjecture~1.3]{VV} lies in showing the existence of smooth \emph{primitive} covers with prescribed scrollar invariants. In \Cref{sec:primitive} we sketch how \Cref{thm:main-thm} can be to prove the existence of primitive curves with some specified scrollar invariants in the case $d = 6$.

\subsubsection{Maroni loci of larger than expected dimension}
Let $\mathcal{H}_{d,g}$ be the moduli stack of degree $d$ covers of $\PP^1$ by genus $g$ smooth irreducible curves. Define the \emph{Maroni locus} 
\[
\mathcal{M}(\vec{e}) = \{\pi : C \rightarrow \PP^1 : \vec{e}_{\pi} = \vec{e} \} 
\subseteq \mathcal{H}_{d,g}.
\]
An application of Serre duality shows that $\sum e_i = d + g - 1$, so henceforth assume this. If the Tschirnhausen bundle of a cover behaves like a ``random'' vector bundle on $\PP^1$ of rank $d-1$ and degree $d + g - 1$, one might expect that:
\begin{enumerate}[label=(\arabic*)]
    \item \label{iii:expect1} 
    $\mathcal{M}(\vec{e})$ is nonempty precisely when $h^1(\PP^1, \End(\cO(\vec{e}))) \leq \dim(\mathcal{H}_{d,g})$; and
    \item \label{iii:expect2}
    when $\mathcal{M}(\vec{e})$ is nonempty, we have $\codim \mathcal{M}(\vec{e}) = h^1(\PP^1, \End(\cO(\vec{e})))$.
\end{enumerate}
Unfortunately, both of these guesses are false in general. Regarding \ref{iii:expect1}, the algebra structure on $\pi_* \cO_C$ sometimes forces $\mathcal{M}(\vec{e})$ to be empty when $h^1(\PP^1, \End(\cO(\vec{e}))) \leq \dim(\mathcal{H}_{d,g})$ (for example take $d = 3$ and $\vec{e} = (1,3)$). Regarding \ref{iii:expect2}, in the cases $d = 4,5$ it is known that the dimension of $\mathcal{M}(\vec{e})$ can be \emph{larger} than the expected dimension when the elements of $\vec{e}$ are spread out; see \cite{VV, FV}. 

Let $\mathcal{M}^{\fac}(\vec{e}) \subseteq \mathcal{M}(\vec{e})$ be the locus of covers factoring through a nontrivial proper subcover. In \Cref{sec:triple-double} and \Cref{sec:double-triple}, we explicitly construct elements of $\mathcal{M}^{\fac}(\vec{e})$ when possible. We do not compute $\dim \mathcal{M}^{\fac}(\vec{e})$, but in many cases such a computation should be possible. Now, the inequality
\[
    \dim \mathcal{M}(\vec{e}) \geq \dim \mathcal{M}^{\fac}(\vec{e})
\]
yields a \emph{lower bound} on the dimension of $\mathcal{M}(\vec{e})$. In cases where the elements of $\vec{e}$ are spread out, we expect that
\[
    \dim \mathcal{M}^{\fac}(\vec{e}) \geq \dim(\mathcal{H}_{6,g}) - h^1(\PP^1, \End(\cO(\vec{e}))).
\]
Combining these two inequalities, one may be able to show that when $d = 6$ many Maroni loci have dimension larger than the expected dimension.

\subsubsection{Making a conjecture for the Tschirnhausen realization problem for all $d$}
It is natural to wonder if, for general degree $d$, all constraints on the scrollar invariants $\vec{e}_\pi$ ``come from'' multiplication over some (possibly trivial) subcover of $\pi$. If this is true then it may be possible to combine the ideas of \Cref{sec:necessary} with \cite[Section~1.3]{Vemulapalli:bounds} to generalize \cite[Conjecture~1.3]{VV} to give a conjectural answer to the Tschirnhausen realization problem for composite $d > 6$. This would be a substantial step forward towards the Tschirnhausen realization problem.

\subsubsection{Solving the Tschirnhausen realization problem in degrees $8$ and $9$}
Degrees $8$ and $9$ provide excellent test cases for examining when all constraints on the Tschirnhausen bundle ``come from'' multiplication. Using the ideas of \Cref{sec:necessary} and \cite{Vemulapalli:bounds}, one can extract a number of constraints on the Tschirnhausen bundles of degree $8$ and $9$ covers explicitly. \emph{If} these conditions are sufficient for the existence of a cover, the natural question is now to construct all possibilities. 

We now present two mechanisms by which one might construct degree $8$ covers with specified scrollar invariants. Let $\tau \colon X \rightarrow \PP^1$ be a hyperelliptic curve. First, one may specify a quartic cover of $X$ by specifying a double cover $\gamma \colon C \rightarrow X$ along with a line bundle $L$ on $C$ and a squarefree section $s \in L^{\otimes -2}$. A theorem of Wood \cite[Theorem~1.4]{ideal-classes} describes\footnote{Taking $U = V$, $S = X$ and $n = 2$ in \cite[Theorem~1.4]{ideal-classes} and appropriately modifying to the case $\Sym^2 V \otimes W$.} the moduli space of pairs $(\gamma \colon C \rightarrow X, L)$ where $\gamma$ is a double cover, $L$ is a line bundle on $C$, and the pushforwards $\gamma_* \cO_C$, $\gamma_*L$ and $\gamma_* (L^{\otimes -2})$ are specified.  The moduli space is described as a quotient of an open locus of sections of a vector bundle on $X$ by a group. Upon fixing the pushforwards, one needs to check the nonemptiness of this open locus. Alternatively, one could try to use the parametrization of quartic algebras \cite{CE} over a hyperelliptic curve.

\begin{remark}
    A related application of Wood's parametrization to covers of $\PP^1$ is given in \cite{LV}, which uses Wood's parametrization to study the Brill--Noether loci of covers in Hirzebruch surfaces.   
\end{remark}

In degree $9$, one may construct examples as triple covers of triple covers. Let $\tau \colon X \rightarrow \PP^1$ be a triple cover of $\PP^1$ and let $V$ be a rank $2$ vector bundle on $X$. It is well-known that the moduli stack of triple covers $\gamma \colon C \rightarrow X$ with $(\gamma_* \cO_C/\cO_X)^{\vee} \simeq V$ is given by the quotient of an open locus of $\Sym^3 V \otimes \det(V)^{\vee}$ by $\Aut(V)$. One may construct examples where $V$ is a direct sum of line bundles; see \cite{Larson:BNT} for a description of line bundles on trigonal curves and their pushforwards to $\PP^1$.

\subsection{History}
There has been a great deal of work on the Tschirnhausen realization problem. We highlight here two ways (out of many) of constructing curves with prescribed scrollar invariants. In \cite[Theorem~1.5]{VV} a construction is given for a cover with specified scrollar invariants whenever $\vec{e}$ is \emph{concave}; the proof involves desingularizing curves in Hirzebruch surfaces and keeping track of how the scrollar invariants change. In \cite[Theorem~1.1]{dp} it is shown that for any $\vec{e} = (e_1,\dots,e_{d-1})$, there exists an integer $k \in \ZZ_{\geq 0}$ such that $(e_1 +k,\dots,e_{d-1}+k)$ is realized as the scrollar invariants of some cover; the proof involves deforming curves in projective bundles. For further constructions of curves with specified scrollar invariants, see \cite{kato-ohbuchi,coppens,ballico,coppens2,cfz, redigolo}. For constraints on scrollar invariants of curves in general degree, see \cite{oh-bounds,Vemulapalli:bounds,VV}.

\subsection{Conventions}
We work over the complex numbers, but all results in this paper hold over an algebraically closed field of characteristic zero. All curves are assumed to be smooth, projective, and irreducible unless otherwise specified. Throughout we drop the subscript and write $\cO = \cO_{\PP^1}$ for the structure sheaf of $\PP^1$. If $X$ is a curve, then a degree $d$ cover of $X$ means a finite morphism $C \rightarrow X$ of degree $d$, where $C$ is smooth and irreducible.

\subsection{Acknowledgements}
The authors would like to thank Ravi Vakil and Melanie Matchett Wood for helpful discussions. SF was supported by the HYPERFORM consortium, funded by France through Bpifrance. SV was supported by the National Science Foundation under grant number DMS2303211.

\section{Constraints on the Tschirnhausen bundle}
\label{sec:necessary}

We begin with a warm-up showing how multiplication imposes constraints on the Tschirnhausen bundle of a trigonal cover.

\begin{example}
\label{ex:multiplication-constraints}
Suppose that $\pi \colon C \to \PP^1$ has degree $d=3$. We claim that $e_2 \leq 2e_1$. 
Let 
\[m \colon \pi_* \cO_C \otimes \pi_* \cO_C \rightarrow \pi_* \cO_C
\]
denote the multiplication map. Fix a splitting $\pi_* \cO_C = \cO \oplus \cO(-e_1) \oplus \cO(-e_2)$ and write $\mathcal{A} = \cO \oplus \cO(-e_1)$. If $m( \mathcal{A} \otimes \mathcal{A} ) \subseteq \mathcal{A}$ then $\mathcal{A}$ is a quadratic subalgebra of $\pi_* \cO_C$, and therefore the trigonal cover $\pi \colon C \rightarrow \PP^1$ factors through the degree $2$ morphism $\relSpec_{\PP^1}(\mathcal{A}) \rightarrow \PP^1$, which is a contradiction. Therefore, the multiplication map $m$ induces some nonzero morphism $\cO(-e_1) \otimes \cO(-e_1) \rightarrow \cO(-e_2)$ (composing $m$ with the projection onto $\cO(-e_2)$) and thus $e_2 \leq 2e_1$. 
\end{example}

In the remainder of this section, we will prove the following theorem (which is the ``only if'' direction of \Cref{thm:main-thm}).

\begin{theorem}
\label{thm:containment}
Let $\pi : C \rightarrow \PP^1$ be a sextic cover with scrollar invariants $\vec{e}_\pi = (e_1,\dots,e_5)$. Then $\vec{e}_\pi \in \mathcal{Q} \cup \mathcal{P}_3$.
\end{theorem}
The bounds in \cite[Theorem 1.26]{Vemulapalli:bounds} imply that $\vec{e}_\pi \in \mathcal{P}_2 \cup \mathcal{P}_3$. Therefore to prove \Cref{thm:containment} it is sufficient to prove the following proposition.

\begin{proposition}
\label{prop:constraints}
If $\vec{e}_\pi \in \mathcal{P}_2 \setminus \mathcal{P}_3$, then $\vec{e}_\pi \in \mathcal{Q}$.
\end{proposition}

We first give a brief outline of the proof of \Cref{prop:constraints} before giving the proof in \Cref{sec:prop-proof}. Let $\pi \colon C \rightarrow \PP^1$ be a sextic cover with $\vec{e}_\pi \in \mathcal{P}_2 \setminus \mathcal{P}_3$. \Cref{lem:factoring} shows that $\pi$ factors as $\tau \circ \gamma$ through a degree $2$ cover $\tau \colon X \to \PP^1$; consider the vector bundle $\gamma_* \cO_C$. This vector bundle splits as $\cO_X \oplus {V}$ for some rank $2$ vector bundle $V$ on $X$. The vector bundle $V$ is either semistable or unstable.

If ${V}$ is semistable, \Cref{lem:semistable-splitting} bounds the difference of summands of $\tau_* V$. In \Cref{lem:containment-stable} we give an explicit combinatorial construction of an admissible partition to conclude that $\vec{e}_\pi \in \mathcal{Q}$.

If ${V}$ is unstable then its Harder--Narasimhan filtration has the form $0 \subset {L} \subset {V}$ for a line bundle $L$ on $X$. In this case, the exact sequence of vector bundles $0 \rightarrow L \rightarrow  V \rightarrow V/L \rightarrow 0$ pushes forward to an exact sequence of vector bundles 
\begin{equation}
\label{eqn:potentially-nonsplit}
0 \rightarrow \tau_* L \rightarrow  \tau_*V \rightarrow \tau_*(V/L) \rightarrow 0.
\end{equation}
It is possible that sequence \eqref{eqn:potentially-nonsplit} does not split. Nevertheless, using properties of the Harder--Narasimhan filtration and the $\cO_X$-algebra structure of $\gamma_* \cO_C$, we directly prove several linear inequalities on the summands of $\tau_* L$ and $\tau_*(V/L)$. Now, explicit manipulation of these linear inequalities (\Cref{lem:extensions}) shows that $\vec{e}_\pi \in \mathcal{Q}$. 

\subsection{Proof of \texorpdfstring{\Cref{prop:constraints}}{Proposition 2.3}} \label{sec:prop-proof}
In order to prove \Cref{prop:constraints} we begin with several lemmas. We first prove that those sextic covers whose scrollar invariants are \emph{not} contained in $\mathcal{P}_3$ factor through quadratic subcovers.

\begin{lemma}[Factoring through a quadratic subcover]
\label{lem:factoring}
Let $\pi \colon C \rightarrow \PP^1$ be a sextic cover with scrollar invariants $\vec{e}_\pi$. If $\vec{e}_\pi \in \mathcal{P}_2 \setminus \mathcal{P}_3$, then $\pi$ factors through a hyperelliptic curve
\[
\begin{tikzcd}
    C \arrow{r}{\gamma} \arrow[swap]{dr}{\pi} & X \arrow{d}{\tau} \\
     & \PP^1
  \end{tikzcd}
\]
of genus $g_X = e_1 - 1$. 
\end{lemma}
\begin{proof}
Fix a splitting $\pi_* \cO_C = \cO \oplus \cO(-e_1) \oplus \dots \oplus \cO(-e_5)$ and consider the multiplication map
\[
    m \colon (\pi_* \cO_C) \otimes (\pi_* \cO_C) \rightarrow (\pi_* \cO_C)
\]
We claim that $\mathcal{A} = \cO \oplus \cO(-e_1)$ is a quadratic subalgebra of $\pi_* \cO_C$, after which the lemma follows immediately upon setting $X = \relSpec_{\PP^1}\mathcal{A}$. 
First note that because $\vec{e}_\pi \in \mathcal{P}_2 \setminus \mathcal{P}_3$, either $e_2 > 2e_1$ or $e_4 > e_1 + e_3$. Now if $e_2 > 2e_1$ then
\[
    m(\mathcal{A} \otimes_k \mathcal{A}) \subseteq \mathcal{A}
\]
and it follows that $\mathcal{A}$ is a quadratic subalgebra of $\pi_* \cO_C$.

Similarly if $e_4 > e_1 + e_3$, then letting ${M} = \cO \oplus \dots \oplus \cO(-e_3)$ we have 
\begin{equation}
\label{eqn:inclusion}
m\left (\mathcal{A} \otimes_k {M}\right) \subseteq {M}.
\end{equation}
Let $K$ and $k$ be the function fields of $C$ and $\PP^1$ respectively, and identify $K$ as a $k$-algebra with the fiber of $\pi_* \cO_C$ at the generic point of $\PP^1$. Denote by $\mathcal{A}_\eta$ and $M_\eta$ the fibers of $\mathcal{A}$ and ${M}$ at the generic point $\eta$ of $\PP^1$. Note that $\mathcal{A}_\eta$ and $M_\eta$ are $k$-subvector spaces of $K$ of dimension $2$ and $4$ respectively and that the specialization $m \colon \mathcal{A}_\eta \otimes_k M_\eta \rightarrow K$ is multiplication in the field $K$. But \eqref{eqn:inclusion} implies that $m(\mathcal{A}_\eta \otimes M_\eta) = M_\eta$, and so $M$ is a vector space over the subfield $k(\mathcal{A}_\eta) \subset K$. It follows that $k(\mathcal{A}_\eta)/k$ has degree dividing $2$ since $\dim_k M_\eta = 4$ and $\dim_k K = 6$. But $\deg k(\mathcal{A}_\eta) \geq \dim_k \mathcal{A}_\eta = 2$ and therefore $\mathcal{A}_\eta = k(\mathcal{A}_\eta)$ which implies that $\mathcal{A}$ is a quadratic subalgebra of $\pi_* \cO_C$. 
\end{proof}

For a vector bundle $V$ on a curve, let $\lambda(V) = \deg(V)/\rank(V)$ denote the slope of $V$. In the case when $V$ is semistable the following \Cref{lem:semistable-splitting,lem:containment-stable} will be used to deduce \Cref{prop:constraints}.

\begin{lemma}
\label{lem:semistable-splitting}
Let $\pi \colon C \rightarrow \PP^1$ be a degree $d$ cover of the projective line by a genus $g$ curve, and let ${E}$ be a semistable vector bundle of rank $r$ on $C$. Write
\[
    \pi_* {E} \simeq \bigoplus_{i = 1}^r \cO(a_i)
\]
for $a_1 \leq \dots  \leq a_r$. Then $a_r - a_1 \leq \frac{2}{d}( g+d - 1)$.
\end{lemma}

\begin{remark}
    We note that the result \cite[Theorem~1.6]{BSTTTZ_bounds} is the arithmetic analogue of \Cref{lem:semistable-splitting} (in the case where $E = \cO_C$) under the dictionary between degree $d$ number fields $K/\QQ$ and degree $d$ covers $C \to \PP^1$ (see~\cite[Section~8]{VV} for a precise description of this dictionary).
\end{remark}

\begin{proof}[Proof of \Cref{lem:semistable-splitting}]
Fix a splitting $\pi_* {E} = \bigoplus_{i = 1}^r \cO(a_i)$. Since $\pi$ is flat, the inclusion $\cO(a_r) \subseteq \pi_* {E}$ induces an inclusion $\pi^* \cO(a_r) \subseteq \pi^* \pi_* {E}$. Composing this with the projection formula $\pi^* \pi_* {E} \rightarrow {E}$ yields a morphism
\[
  \phi \colon \pi^* \cO(a_r) \rightarrow {E}.
\]

We first claim that $\phi$ is nonzero. Let $K$ and $k$ be the function fields of $C$ and $\PP^1$ respectively. The fiber ${E}_\xi$ of $E$ at the generic point $\xi \in C$ is an $r$-dimensional $K$-vector space. The morphism $\pi^* \pi_* {E} \rightarrow {E}$ arising from the projection formula specializes to the multiplication map $K \otimes_k E_\xi \rightarrow E_\xi$ given by $x \otimes y \mapsto xy$. The fiber of $\cO(a_r)$ at the generic point $\eta \in \PP^1$ is a $1$-dimensional $k$-vector subspace $W \subset E_\xi$. Therefore on the generic point, $\phi$ specializes to the nonzero linear map $K \otimes_k W \rightarrow E_\xi$ sending $x \otimes y \mapsto xy$. Therefore, $\phi$ is nonzero.

A nonzero morphism from a line bundle to a vector bundle is an inclusion and therefore $\phi(\pi^* \cO(a_r))$ is a subbundle of ${E}$. Since ${E}$ is semistable, we have 
\[
  \lambda(\phi(\pi^* \cO(a_r))) \leq \lambda({E}).
\]
On the other hand
\[
  \lambda(\phi(\pi^* \cO(a_r))) \geq \lambda(\pi^* \cO(a_r)) = da_r.
\]
Combining the above two inequalities, we obtain
\begin{equation}
  \label{eqn:mu-E-lower}
  da_r \leq \lambda({E}).
\end{equation}

Let $\omega_{C/\PP^1}$ be the relative dualizing sheaf of $\pi \colon C \to \PP^1$. Since $E$ is semistable and $\omega_{C/\PP^1}$ is a line bundle the bundle ${E}^{\vee} \otimes \omega_{C/\PP^1}$ is semistable. Relative Serre duality implies that
\[
  \pi_*( {E}^{\vee} \otimes \omega_{C/\PP^1}) \simeq (\pi_* {E})^{\vee}.
\]
A similar argument to the above shows that $-da_1 \leq \lambda(\pi_* E)^\vee = \lambda({E}^{\vee} \otimes \omega_{C/\PP^1})$. Since the slope is additive over the tensor product with a line bundle we have
\begin{equation}
  \label{eqn:bound-with-omega}
  -da_1 \leq -\lambda({E}) + \deg(\omega_{C/\PP^1}).
\end{equation}
Taking degrees in the relation $\omega_{C/\PP^1} = \omega_C \otimes (\pi^* \omega_{\PP^1})^{\vee}$ yields 
\begin{equation}
  \label{eqn:rel-can-deg}
  \deg(\omega_{C/\PP^1}) = 2g - 2 + 2d .
\end{equation}
Combining \eqref{eqn:bound-with-omega} and \eqref{eqn:rel-can-deg} we obtain
\begin{equation}
  \label{eqn:mu-E-upper}
  \lambda({E}) \leq da_1 + 2(g + d- 1).
\end{equation}
and combining \eqref{eqn:mu-E-lower} and \eqref{eqn:mu-E-upper} gives $a_r - a_1 \leq \frac{2}{d}(g + d- 1)$, as required.
\end{proof}

\begin{lemma}
  \label{lem:containment-stable}
  Suppose that $\vec{e} \in \mathcal{P}_2 \setminus \mathcal{P}_3$ has $e_5 \leq e_1 + e_2$. Then $\vec{e} \in \mathcal{Q}$.
\end{lemma}
\begin{proof}
Because $\vec{e} \in \mathcal{P}_2 \setminus \mathcal{P}_3$, we must have either $e_2 > 2e_1$ or $e_4 > e_1 + e_3$. By assumption $e_5 \leq e_1 + e_2$ so $e_4 \leq e_1 + e_3$ and hence $e_2 > 2e_1$. 

It suffices to show that $\{2,3\} \sqcup \{4,5\}$ is an admissible partition for $\vec{e}$ (in the sense of \Cref{def:admissible-partition}). Conditions~\ref{cond:1}, \ref{cond:2}, and \ref{cond:3} of \Cref{def:admissible-partition} follow immediately.

To see Condition~\ref{cond:4}, we claim that $(e_k - e_i) + (e_{\ell} - e_j) + e_1 \leq e_i + e_j$. Since $e_5 \leq e_2 + e_1$ we have $e_k - e_i \leq e_1$ and $e_{\ell} - e_j \leq e_1$. Moreover, $e_2 > 2e_1$ gives $4e_1 < e_i + e_j$. Therefore
\[
  (e_k - e_i) + (e_{\ell} - e_j) + e_1 \leq 3e_1 < e_i + e_j
\]
as required.
\end{proof}

\Cref{lem:semistable-splitting,lem:containment-stable} will be sufficient to deduce \Cref{prop:constraints} when $V$ is semistable. When $V$ is unstable we will need the following lemma instead. 

\begin{lemma}
  \label{lem:extensions}
  Let $\vec{e} \in \mathcal{P}_2 \setminus \mathcal{P}_3$ and write $E = \cO(e_2) \oplus \cO(e_3) \oplus \cO(e_4) \oplus \cO(e_5)$. Let
  \begin{equation}
    \label{eqn:exact-seq}
    0 \rightarrow F \rightarrow E \rightarrow G \rightarrow 0
  \end{equation}
  be an exact sequence of vector bundles on $\PP^1$ with $\rank(F) = \rank(G) = 2$ and  $\lambda(G) \leq \lambda(F)$. Write $F = \cO(f_1) \oplus \cO(f_2)$ and $G = \cO(g_1) \oplus \cO(g_2)$ for integers $f_1 \leq f_2 \leq e_1 + f_1$ and $g_1 \leq g_2 \leq e_1 + g_1$. If $f_1+f_2 + e_1 \leq 2g_1 + 2g_2$ then $\vec{e} \in \mathcal{Q}$.
\end{lemma}
\begin{proof}
The exactness of the sequence \eqref{eqn:exact-seq} implies that
\begin{equation}
\label{eqn:equality}
    g_1 + g_2 + f_1 + f_2 = e_2 + e_3 + e_4 + e_5.
\end{equation}
The exact sequence \eqref{eqn:exact-seq} gives an extension class $\alpha \in \Ext^1(G,F)$. Consider the vector bundle $\mathcal{E} \rightarrow \PP^1 \times \AA^1$ where the restriction to $\PP^1 \times \{t\}$ is given by the extension class $t\alpha \in \Ext^1(G,F)$. Because all nonzero scalar multiples of $\alpha$ give rise to isomorphic vector bundles if $t \neq 0$, we have an isomorphism $\mathcal{E}_t \simeq E$. On the other hand, the special fiber at $t = 0$ is $\mathcal{E}_0 \simeq F \oplus G$. 

Now, the upper semi-continuity theorem implies that for any integer $n$ we have
\begin{equation}
\label{eqn:semicontinuity}
     h^0(\PP^1, E(n)) \leq h^0(\PP^1, (F \oplus G)(n)).
\end{equation}
Write $F \oplus G = \cO(c_1) \oplus \dots \oplus \cO(c_4)$ ordered so that $c_1 \leq \dots \leq c_4$. From \eqref{eqn:semicontinuity}, a standard argument now implies that
\begin{equation}
  \label{eqn:two-summands}
  c_1 \leq e_2, \quad \text{ and } \quad c_1 + c_2 \leq e_2 + e_3, \quad \text{ and } \quad c_1 + c_2 + c_3 \leq e_2 + e_3 + e_4.
\end{equation}

There are four possibilities for $(c_1,c_2,c_3,c_4)$. Either
\begin{enumerate}[label=(\roman*)]
\item \label{casei}
  $(c_1,c_2,c_3,c_4) = (g_1,g_2,f_1,f_2)$,
\item \label{caseii}
  $(c_1,c_2,c_3,c_4) = (g_1,f_1,f_2,g_2)$,
\item \label{caseiii}
  $(c_1,c_2,c_3,c_4) = (f_1,g_1,g_2,f_2)$,
\item \label{caseiv}
  $(c_1,c_2,c_3,c_4) = (g_1,f_1,g_2,f_2)$.
\end{enumerate}

In case~\ref{casei}, combining \cref{eqn:equality,eqn:two-summands} gives
\[
    e_4 + e_5 + e_1 \leq f_1+f_2 + e_1 \leq 2g_1 + 2g_2 \leq e_2 + e_3,
\]
and thus $\{2, 3\} \sqcup \{4, 5\}$ is an admissible partition for $\vec{e}$, as required. 

In case~\ref{caseii}, we have $g_1 \leq f_1 \leq f_2 \leq g_2$, so $c_4 - c_1 = g_2 - g_1 \leq e_1$. This implies $e_5 - e_2 \leq e_1$ and \Cref{lem:containment-stable} implies that $\vec{e} \in \mathcal{Q}$. An analogous argument holds in case~\ref{caseiii}.

In case~\ref{caseiv}, we claim that $\{2, 4\} \sqcup \{3, 5\}$ is an admissible partition. The fact that $E \rightarrow G$ is surjective implies that $e_2 \leq g_1$. Combining this with \eqref{eqn:two-summands} implies that $e_2 = g_1$. Dualizing the exact sequence \eqref{eqn:exact-seq} and applying the same argument implies that $e_5 = f_2$. Therefore we have
\begin{equation}
\label{eqn:order}
 e_2 = g_1 \leq f_1 \leq e_3 \leq e_4 \leq g_2 \leq f_2 = e_5.
\end{equation}
The inequality $g_2 - g_1 \leq e_1$ implies $e_4 - e_2 \leq e_1$, and similarly $f_2 - f_1 \leq e_1$ implies $e_5 - e_3 \leq e_1$.  Because $\vec{e} \in \mathcal{P}_2 \setminus \mathcal{P}_3$, we have either $e_4 > e_1 + e_3$ or $e_2 > 2e_1$. From \cref{eqn:order}, we have $e_4 - e_3 \leq g_2 - g_1 \leq e_1$, so indeed $e_2 > 2e_1$. Now
\[
    (e_5 - e_4) + (e_3 - e_2) + e_1 \leq 3e_1 \leq e_2 + e_3
\]
and therefore $\{2, 4\} \sqcup \{3, 5\}$ is an admissible partition, as required.
\end{proof}

\begin{remark}
\Cref{lem:extensions} also follows from \cite[Section~2]{lin}, which gives a combinatorial description of short exact sequences of vector bundles on $\PP^1$.
\end{remark}

\begin{proof}[Proof of \Cref{prop:constraints}]
Suppose that $\vec{e}_\pi \in \mathcal{P}_2 \setminus \mathcal{P}_3$ is the scrollar invariants of a sextic cover $\pi\colon C \rightarrow \PP^1$. By \Cref{lem:factoring} the morphism $\pi$ factors as 
\[
\begin{tikzcd}
    C \arrow{r}{\gamma} \arrow[swap]{dr}{\pi} & X \arrow{d}{\tau} \\
     & \PP^1
  \end{tikzcd}
\]
through a hyperelliptic curve $X$ of genus $g_X = e_1 - 1$. Fix a splitting $\gamma^* \cO_X \simeq \cO_{X} \oplus V$ for some rank $2$ vector bundle $V$ on $X$. We have that
\[
  \tau_* V = \cO(-e_2) \oplus \dots \oplus \cO(-e_5),
\]
If $V$ is semistable, then \Cref{lem:semistable-splitting} shows that $e_5 \leq e_2 + e_1$ and by \Cref{lem:containment-stable} we have $\vec{e}_\pi \in \mathcal{Q}$.

We may therefore assume that $V$ is unstable, in which case it admits a Harder--Narasimhan filtration of the form $0 \subset L \subseteq V$ for a line bundle $L$ on $X$. Pushing forward and dualizing the exact sequence $0 \rightarrow L \rightarrow V \rightarrow V/L \rightarrow 0$ we obtain
\begin{equation}
\label{eqn:dual-exact}
0 \rightarrow (\tau_*(V/L))^{\vee} \to (\tau_* V)^{\vee} \rightarrow ( \tau_* L)^{\vee} \rightarrow 0.
\end{equation}
By construction $\gamma_* \cO_C$ is a cubic $\cO_{X}$-algebra and is thus equipped with a multiplication map $m \colon \gamma_* \cO_C \otimes \gamma_* \cO_C \rightarrow \gamma_* \cO_C$. Precomposing $m$ with the inclusion $L \subseteq \gamma_* \cO_C$ and postcomposing with the quotient morphism $\gamma_* \cO_C \rightarrow V \rightarrow V/L$ yields a map
\begin{equation}
  \label{eqn:key-mult-coeff}
  \overline{m} \colon L \otimes L \rightarrow V/L .
\end{equation}

Note that if $\overline{m}$ is uniformly zero then $m((\cO_X \oplus L) \otimes (\cO_X \oplus L)) \subseteq \cO_X \oplus L$, in which case $\cO_X \oplus L$ is a quadratic subalgebra of the cubic $\cO_X$-algebra $\gamma_* \cO_C$, which is a contradiction. Therefore $\overline{m}$ is nonzero and therefore
\begin{equation}
  \label{eqn:degree}
  2\deg (L) \leq \deg (V/L).
\end{equation}
Choose splittings $(\tau_* (V/L))^{\vee} \simeq \cO(f_1) \oplus \cO(f_2)$ and $(\tau_* L)^{\vee} \simeq \cO(g_1) \oplus \cO(g_2)$. Now, relative Serre duality implies that for a line bundle $E$ on $X$ we have $\deg(\tau_* E) = \deg(E) + 1 - g_X - \deg(\tau)$. Applying this to $L$ and $V/L$, substituting $g_X = e_1 - 1$ and combining with \eqref{eqn:degree} now yields
\begin{equation}
\label{eqn:mult-constraint}
f_1 + f_2 + e_1 \leq 2g_1 + 2g_2.
\end{equation}
Applying \Cref{lem:extensions} to the exact sequence in \eqref{eqn:dual-exact} we see that $\vec{e} \in \mathcal{Q}$, as required. 
\end{proof}

\section{Double covers of triple covers}
\label{sec:triple-double}

The purpose of this section is to prove the following theorem. 
\begin{theorem}
\label{thm:double-cover-of-cubic}
Every quintuple $\vec{e} \in \mathcal{P}_3$ is realized as the scrollar invariants of a cover $\pi \colon C \to \PP^1$ where $\pi = \tau \circ \gamma$ factors as a double cover of a triple cover (that is $\deg(\tau) = 3$ and $\deg(\gamma) = 2$).
\end{theorem}

The basic strategy is as follows. First, we construct a cubic cover $\tau \colon X \rightarrow \PP^1$ with specified scrollar invariants $(e_1,e_2)$. Then we construct a double cover $\gamma \colon C \rightarrow X$ such that $\gamma_* \cO_C = \cO_X \oplus L$ for some line bundle $L$ on $X$ with $\tau_* L = \cO(-e_3) \oplus \cO(-e_4) \oplus \cO(-e_5)$. 

The key technical difficulty is producing a double cover $\gamma \colon C \rightarrow X$ such that $C$ is smooth and its structure sheaf has the desired pushforward. As a result, the Brill--Noether theory of trigonal covers (\cite{Larson:BNT}) plays an essential role in our proof. 

We first recall the solution to the Tschirnhausen realization problem in degree $d = 3$.

\begin{theorem}[{Maroni~\cite{maroni}}]
  \label{thm:maroni}
  Let $1 \leq e_1 \leq e_2$ be integers. The pair $(e_1,e_2)$ arises as the scrollar invariants of a cubic cover if and only if $e_2 \leq 2e_1$. 
\end{theorem}

\begin{lemma}
  \label{lem:cubic-bpf}
  Let $\tau\colon X \rightarrow \PP^1$ be a cubic cover and let $L \not \simeq \cO_X$ be a line bundle on $X$ such that $L^{\otimes -2}$ is basepoint-free. Then there exists a double cover $\gamma \colon C \rightarrow X$ by a smooth curve such that $\gamma_* \cO_{C} = \cO_X \oplus L$.
\end{lemma}
\begin{proof}
  Choose a generic section $s \in H^0(X,L^{\otimes -2})$. From Bertini's theorem, the basepoint-freeness of $L^{\otimes -2}$ implies that the divisor $Z(s)$, cut out by the vanishing of $s$, is reduced. Define the quadratic $\cO_X$-algebra $\mathcal{A} = \cO_X \oplus L$ with multiplication map $\mathcal{A} \otimes \mathcal{A} \rightarrow \mathcal{A}$ given by
  \[
    (x_1,y_1) \otimes (x_2,y_2) \mapsto (x_1x_2 + sy_1y_2,x_1y_2 + x_2y_1).
  \]

  Writing $C = \relSpec_{X}(\mathcal{A})$ we obtain a finite flat morphism $\gamma \colon C  \rightarrow X$ of degree $2$, and by definition $\gamma_* \cO_C = \mathcal{A}$. The divisor $Z(s)$ is easily seen to be the ramification divisor of the morphism $\gamma$. In particular, since $Z(s)$ is reduced $C$ is smooth.

  Finally, if $C$ were reducible and smooth, then $C \simeq X \sqcup X$ and $\gamma_* \cO_C = \cO_X \oplus \cO_X$. But $L \not \simeq \cO_X$ by assumption, so $C$ is irreducible and hence $\gamma$ is a double cover.
\end{proof}

The key technical proposition is now \Cref{prop:main-cubic-prop} whose proof we delay to \Cref{subsec:proof-of-main-cubic}.

\begin{proposition}
\label{prop:main-cubic-prop}
Let $\vec{e} = (e_1, \dots, e_5) \in \mathcal{P}_3$ and let $\tau \colon X \rightarrow \PP^1$ be a trigonal curve with scrollar invariants $\vec{e}_\tau = (e_1,e_2)$. There exists a line bundle $L$ on $X$ such that:
\begin{enumerate}[label=(\arabic*)]
    \item \label{iii:main-cubic1} $\tau_* L = \cO(-e_3) \oplus \cO(-e_4) \oplus \cO(-e_5)$; and
    \item \label{iii:main-cubic2} $L^{\otimes -2}$ is basepoint-free.
\end{enumerate}
\end{proposition}

Assuming \Cref{prop:main-cubic-prop}, we now give a proof of \Cref{thm:double-cover-of-cubic}.

\begin{proof}[Proof of Theorem~\ref{thm:double-cover-of-cubic}]
Let $\vec{e} = (e_1,\dots,e_5) \in \mathcal{P}_3$. By \Cref{thm:maroni}, there exists a trigonal curve $\tau \colon X \rightarrow \PP^1$ with scrollar invariants $\vec{e}_\tau = (e_1,e_2)$. By \Cref{prop:main-cubic-prop}, there exists a line bundle $L$ on $X$ satisfying \ref{iii:main-cubic1} and \ref{iii:main-cubic2}.

Now by \Cref{lem:cubic-bpf} there exists a double cover $\gamma \colon C \rightarrow X$ with $\gamma_* \cO_C = \cO_X \oplus L$. Thus $\pi = \tau \circ \gamma$ is a smooth irreducible sextic cover with scrollar invariants $\vec{e}_\pi = \vec{e}$.
\end{proof}

\subsection{Proof of Proposition~\ref{prop:main-cubic-prop}}
\label{subsec:proof-of-main-cubic}

We begin by recalling some basic properties of \emph{splitting types} of line bundles on finite covers of the projective line. 

\begin{defn}
Given a degree $d$ cover $\tau \colon X \rightarrow \PP^1$ and a line bundle $L$ on $X$, the \emph{splitting type} of $L$ is the unique $d$-tuple $\vec{a}(L) = (a_1,\dots,a_{d})$ such that $a_1 \leq \dots \leq a_d$ and $\tau_* L \simeq \cO(a_1) \oplus \dots \oplus \cO(a_d)$.
\end{defn}

\begin{defn}
Given a degree $d$ cover $\tau \colon X \rightarrow \PP^1$, and a $d$-tuple of integers $\vec{a} = (a_1,\dots,a_d)$ with $a_1 \leq \dots \leq a_d$, define
\[
    U^{\vec{a}}(X) = \{L \in \Pic(X) : \vec{a}(L) = \vec{a}\}.
\]
\end{defn}

\begin{lemma}
  \label{lem:bpf-criterion}
  Let $\tau \colon X \rightarrow \PP^1$ be a degree $d$ cover and let $L$ be a line bundle on $X$. If $L$ has splitting type $\vec{a}(L) = (a_1, \dots, a_d)$ with $a_1 \geq 0$, then $L$ is basepoint-free. 
\end{lemma}
\begin{proof}
  For any closed point $P \in X$, the inclusion $L(-P) \xhookrightarrow{ } L$ pushes forward to an inclusion $\tau_*L(-P) \xhookrightarrow{ } \tau_*L$. Since $\deg(L(-P)) = \deg(L) - 1$ the splitting type of $L(-P)$ is given as $\vec{a}(L(-P)) = (a_1, \dots, a_{i-1}, a_i - 1, a_{i + 1}, \dots, a_d)$ for some $1 \leq i \leq d$. We then have
  \begin{align*}
    h^0(X,L(-P)) &= h^0(\PP^1, \tau_* L(-P))                                                          \\
                 &= h^0(\PP^1,\cO(a_i - 1)) + \sum_{j \neq i} h^0(\PP^1, \cO(a_j))                    \\
                 &= -1 + \sum_{j = 1}^d h^0(\PP^1, \cO(a_j)) & & \text{because $0 \leq a_1 \leq a_i$} \\
                 &= -1 + h^0(\PP^1,\tau_* L)                                                          \\
                 &= h^0(\PP^1, L).
  \end{align*}
  Therefore $L$ is basepoint-free.
\end{proof}

For the rest of this section, we assume that $\vec{e} \in \mathcal{P}_3$ so that (by \Cref{thm:maroni}) we may fix a cubic cover $\tau \colon X \rightarrow \PP^1$ with scrollar invariants $\vec{e}_\tau = (e_1,e_2)$. The genus of $X$ is $g_X = e_1 + e_2 - 2$. A degree calculation shows that a line bundle $L$ in $U^{\rest}(X)$ has degree $\delta \colonequals g_X + 2 -e_3 - e_4 - e_5$.
Define the map
\[
  \phi \colon U^{\rest}(X) \rightarrow \Pic^{-2\delta}(X)
\]
given by $L \mapsto L^{\otimes -2}$. Define
\[
  U^{\bad}(X) \colonequals \{L \in \Pic^{-2\delta}(X) : \vec{a}(L) = (a_1, \dots, a_d) \text{ with } a_1 < 0\}.
\]
That is, $U^{\bad}(X)$ consists of those line bundles $L \in \Pic^{-2\delta}(X)$ whose pushforwards $\tau_* L$ contain a negative degree summand.

\begin{lemma}
\label{lem:dim-reduction}
If $\dim U^{\rest}(X) > \dim U^{\bad}(X)$, then there exists a line bundle $L$ with splitting type $\vec{a}(L) = (-e_5, -e_4, -e_3)$ such that $L^{\otimes -2}$ is basepoint-free.
\end{lemma}
\begin{proof}
  Let $L,L'$ be line bundles such that $L^{\otimes -2} \simeq (L')^{\otimes -2}$. Then $L^\vee \otimes L' \in \Pic^0(X)$ is $2$-torsion, and therefore the fibers of $\phi$ have cardinality at most $\# \Pic^0(X)[2] = 2^{2g_X}$. In particular, if $\dim U^{\rest}(X) > \dim U^{\bad}(X)$, then there exists $L \in U^{\rest}(X)$ such that $\phi(L) \notin U^{\bad}(X)$. Now Lemma~\ref{lem:bpf-criterion} implies that $\phi(L) = L^{\otimes -2}$ is basepoint-free.
\end{proof}

To verify that $\dim U^{\rest}(X) > \dim U^{\bad}(X)$, we will write $U^{\bad}(X)$ as a finite disjoint union of splitting loci $U^{\vec{a}}(X)$ and then check that $\dim U^{\rest}(X) > \dim U^{\vec{a}}(X)$ for every such $\vec{a}$. In order to do this, we will need to know when splitting loci $U^{\vec{a}}(X)$ are nonempty, and the dimensions of the splitting loci in this case. 

For a sequence $\vec{a} = (a_1,a_2,a_3)$ of nondecreasing integers, define
\[
  \mu(\vec{a}) \colonequals \sum_{i < j} \max\{0,a_j - a_i - 1\} .
\]

\begin{theorem}[{\cite[Theorem~1.1]{Larson:BNT}}]
\label{thm:splitting-type-classification}
Let $\vec{a} = (a_1,a_2,a_3)$ be nondecreasing integers. Then $U^{\vec{a}}(X)$ is nonempty if and only if $a_{i+j} \leq a_i + e_j$ for all $1 \leq i \leq 3$ and $1 \leq j \leq 2$ with $i+j \leq 3$. If $U^{\vec{a}}(X)$ is nonempty, then its dimension is given as follows:
\begin{enumerate}
    \item if $a_3 \leq a_1 + e_1$, then $\dim(U^{\vec{a}}(X)) = g_X - \mu(\vec{a})$;
    \item if $a_1 + e_1 < a_3 < e_2+a_1$, then $\dim(U^{\vec{a}}(X)) = e_2 - (a_3 - a_1)$; and
    \item if $a_3 = e_2+a_1$, then $\dim(U^{\vec{a}}(X)) = 0$.
\end{enumerate}    
\end{theorem}

\begin{coro}
\label{lem:decomposition}
We have $U^{\bad}(X) = \bigsqcup U^{\vec{a}}(X)$, where the disjoint union is over nondecreasing integer sequences $\vec{a} = (a_1,a_2,a_3)$ satisfying:
\begin{enumerate}
    \item $a_{i+j} \leq a_i + e_j$ for all $1 \leq i \leq 3$ and $1 \leq j \leq 2$ with $i+j \leq 3$;
    \item $a_1 + a_2 + a_3 = -2 \delta - g_X - 2$; and
    \item $a_1 \leq -1$.
\end{enumerate}
\end{coro}
\begin{proof}
Follows immediately from \Cref{thm:splitting-type-classification}.
\end{proof}

\begin{lemma}
\label{lem:dim-calc}
Let $\vec{e} \in \mathcal{P}_3$. If $\vec{a}$ satisfies the conditions of \Cref{lem:decomposition}, then
\[
\dim U^{\rest}(X) > \dim U^{\vec{a}}(X).
\]
\end{lemma}
\begin{proof}
The values of $\vec{e}$ and $\vec{a}$ are restricted by the linear inequalities coming from the definition of $\mathcal{P}_3$ and the conditions of \Cref{lem:decomposition}. Our aim is to show that the minimum value of the piecewise linear function $\dim U^{\rest}(X) - \dim U^{\vec{a}}(X)$ is positive on the polytope of possible $\vec{e}$ and $\vec{a}$. Hence, our problem has been reduced to a linear (integer) programming problem. To minimize this piecewise linear function on this region, the authors used two methods. First we used both the \textsc{QSopt\_ex} library~\cite{lin-prog2,lin-prog1,qsopt_ex} and the \textsc{GNU GLPK} library~\cite{glpk} (in particular \texttt{glp\_exact})\footnote{In \cite{electronic} we provide the \texttt{*.lp} and \texttt{*.mps} files which allow the user to verify our claims with most linear programming toolkits.} for rigorous linear programming with rational coefficients. In addition we generated \textsc{Python} code via Gemini Pro 3.1 (utilizing \textsc{SciPy} and \textsc{NumPy}) which we then verified. Both methods yield the same results. 
\end{proof}

\begin{proof}[Proof of \Cref{prop:main-cubic-prop}]
Consider the map 
\[
    \phi \colon U^{\rest}(X) \rightarrow \Pic^{-2\delta}(X)
\]
sending $L \rightarrow L^{\otimes -2}$. It suffices to show that the image of $\phi$ contains a basepoint-free line bundle. By \Cref{lem:dim-reduction}, it suffices to show that $\dim U^{\rest}(X) > \dim U^{\bad}(X)$. Now, \Cref{lem:decomposition} writes $U^{\bad}(X)$ as a finite disjoint union of spaces $U^{\vec{a}}(X)$. Finally, \Cref{lem:dim-calc} shows that for every $\vec{a}$ satisfying these linear inequalities we have $\dim U^{\rest}(X) > \dim U^{\vec{a}}(X)$, as required.
\end{proof}

\section{Triple covers of double covers}
\label{sec:double-triple}

The purpose of this section is to prove the following theorem.
\begin{theorem}
\label{thm:double-triple}
Every quintuple $\vec{e} \in \mathcal{Q}$ is realized as the scrollar invariants of a cover $\pi \colon C \to \PP^1$ where $\pi = \tau \circ \gamma$ factors as a triple cover of a double cover (that is $\deg(\tau) = 2$ and $\deg(\gamma) = 3$).
\end{theorem}

The basic strategy is similar to the proof of \Cref{thm:double-cover-of-cubic}. Namely, fix $\vec{e} \in \mathcal{Q}$ and let $\{i,j\} \sqcup \{k,\ell\}$ be an admissible partition for $\vec{e}$. First, choose a hyperelliptic curve $\tau \colon X \rightarrow \PP^1$ with scrollar invariant $e_1$ (i.e., a hyperelliptic curve of genus $e_1 - 1$). Next, we choose line bundles $L_1$ and $L_2$ on $X$ with splitting types $(-e_j,-e_i)$ and $(-e_\ell,-e_k)$ respectively. The line bundles will be chosen (carefully) so that there is a cubic cover $\gamma \colon C \rightarrow X$ with $\gamma_* \cO_C = \cO_X \oplus L_1 \oplus L_2$.

\subsection{Constructing triple covers of hyperelliptic curves}

Let $\tau \colon X \rightarrow \PP^1$ be a hyperelliptic curve. In this subsection we discuss how, given appropriate line bundles $L_1$ and $L_2$ on $X$, one may construct a triple cover $\gamma \colon C \rightarrow X$ with $\gamma_* \cO_C = \cO_X \oplus L_1 \oplus L_2$. The key proposition is the following.

\begin{proposition}
\label{prop:main-triple-cover-of-double}
Fix $\vec{e} \in \mathcal{Q}$ and let $\{i,j\} \sqcup \{k,\ell\}$ be an admissible partition for $\vec{e}$. Fix a hyperelliptic curve $\tau \colon X \rightarrow \PP^1$ of genus $g_X = e_1 - 1$.  Suppose $L_1$ and $L_2$ are line bundles on $X$ with splitting types  $(-e_i,-e_j)$ and $(-e_k,-e_\ell)$ respectively. Then there exists a triple cover $\gamma \colon C \rightarrow X$ with $\gamma_* \cO_C = \cO_X \oplus L_1 \oplus L_2$.
\end{proposition}

We first prove the following lemma.

\begin{lemma}
  \label{lem:second-line-bundle-degree}
  Under the assumptions of \Cref{prop:main-triple-cover-of-double} we have $h^0(X,L_1 \otimes L_2^{\otimes -2}) > 0$.
\end{lemma}
\begin{proof}
First note that
\[
    \deg(L_1 \otimes L_2^{\otimes -2}) = \deg(L_1) - 2\deg(L_2).
\]
Because $e_{i} + e_j \leq e_k + e_\ell$, we have $\deg(L_2) \leq \deg(L_1)$. Therefore
\[
    \deg(L_1) - 2\deg(L_2) \geq -\deg(L_1).
\]
A short calculation shows that $-\deg(L_1) = e_i + e_j - e_1 \geq e_1 = g_X + 1$. In particular we have $\deg(L_1 \otimes L_2^{\otimes -2}) \geq g_X + 1$ and so by the Riemann--Roch theorem $h^0(X, L_1 \otimes L_2^{\otimes -2}) > 0$.
\end{proof}

\begin{proof}[Proof of \Cref{prop:main-triple-cover-of-double}]
Write $V = L_1 \oplus L_2$ and consider the vector bundle
\[
    \Sym^3 V^{\vee} \otimes \det(V) \simeq (L_2 \otimes L_1^{\otimes -2}) \oplus L_1^{\vee} \oplus L_2^{\vee} \oplus (L_1 \otimes L_2^{\otimes -2}).
\]
Upon fixing a splitting $V = L_1 \oplus L_2$, a section $f \in \Sym^3 V^{\vee}$ can be written as a binary cubic form
\[
    f(x,y) = f_0x^3 + f_1 x^2y + f_2xy^2 + f_3y^3  
\]
where $f_0 \in H^0(X,(L_2 \otimes L_1^{\otimes -2}))$, $f_1 \in H^0(X, L_1^{\vee})$, $f_2 \in H^0(X, L_2^{\vee})$ and $f_3 \in H^0(X,(L_1 \otimes L_2^{\otimes -2}))$. The vanishing of $f$ gives rise to a variety
\[
  C = Z(f) \subseteq \PP_X(V).
\]

Observe that $h^0(X,L_2 \otimes L_1^{\otimes -2}) > 0$ by assumption and $h^0(X,L_1 \otimes L_2^{\otimes -2}) > 0$ follows from Lemma~\ref{lem:second-line-bundle-degree}. But this implies that the linear series $\Sym^3 V^{\vee} \otimes \det(V)$ is basepoint-free on $\PP_X(V)$. Indeed, for a closed point $p \in X$, the fiber $\PP_X(V)$ above $p$ is isomorphic to $\PP^1$, and the restriction 
\[
  (\Sym^3 V^{\vee} \otimes \det(V)) \vert_{p}
\]
contains the binary cubic forms $ax^3 + by^3$ for arbitrary $a,b \in \CC$. As a result, Bertini's theorem implies that a general section $f$ gives rise to a smooth curve $C$. The natural map $\gamma \colon C \rightarrow X$ is easily seen to be a cubic cover and by construction $\gamma_* \cO_C = \cO_X \oplus V$. 
\end{proof}

\subsection{Line bundles on hyperelliptic curves}
We interlude with a brief discussion of standard facts about line bundles on hyperelliptic curves. Let $\tau \colon X \rightarrow \PP^1$ be a hyperelliptic curve of genus $g$. Let $H \colonequals \tau^* [0:1]$ be a divisor representing the fiber class. The main fact we need is \Cref{prop:splitting-types-of-hyp}, which relates divisors of a certain form to line bundles with fixed splitting type. 

\begin{defn}
\label{def:semireduced-divisor}
A \emph{semireduced divisor} on $X$ is an effective divisor which does not contain any fibers of $\tau$.
\end{defn}

\begin{lemma}
  \label{lem:coh-of-semireduced}
  Let $\tau \colon X \to \PP^1$ be a hyperelliptic curve of genus $g_X$ and let $D$ be a semireduced divisor on $X$ of degree $\deg(D) \leq g_X + 1$. Then we have $h^0(X, D - H) = 0$ where $H = \tau_*[0:1]$.
\end{lemma}
\begin{proof}
We first handle the case where $\deg(D) \leq g_X$. Let $K_X$ be a canonical divisor on $X$. In this case the Riemann--Roch theorem implies
\[
    h^0(X,D) = \deg(D) -g_X + 1 + h^0(X,K_X - D).
\]
Because $X$ is hyperelliptic, the canonical map factors as $ X \rightarrow \PP^1 \rightarrow \PP^{g-1}$ where the second map is given by the line bundle $\cO(g_X-1)$ on $\PP^1$. By definition, we have
\[
    h^0(X, K_X - D) = \dim \{f \in H^0(\PP^1, \cO(g_X - 1)) : \tau^* Z(f) \geq D\}
\] 
where the vanishing set $Z(f)$ is considered with multiplicity. However, since $D$ is semireduced, this is equal to
\[
    \dim \{f \in H^0(\PP^1, \cO(g_X-1)) : Z(f) \geq \tau^* D\}.
\]
In particular $h^0(X, K_X - D) = g_X - \deg(D)$, and so $h^0(X,D) = 1$. Because $h^0(X,H) = 2$, the linear system $\lvert D \rvert$ cannot contain $\lvert H \rvert$ as a sublinear series, and hence $h^0(X, D - H) = 0$, as required.

When $\deg(D) = g_X + 1$ a similar calculation shows that $h^0(X,D) = 2 = h^0(X,H)$. Suppose there exists a nonzero section $s \in H^0(X,D-H)$. Multiplication by $s$ gives an isomorphism $H^0(X,H) \rightarrow H^0(X,D)$. Choose $f \in H^0(X,D)$ such that $Z(f) = D$. Then $f = sh$ for some $h \in H^0(X,H)$. In other words, $D = H + Z(s)$ for some effective divisor $Z(s)$, which contradicts the assumption that $D$ is semireduced.
\end{proof}

\begin{proposition}
  \label{prop:splitting-types-of-hyp}
  Let $L$ be a line bundle on a hyperelliptic curve $\tau \colon X \to \PP^1$ of genus $g_X$. Then $L$ has splitting type $(a,b)$ if and only if $L \simeq \tau^* \cO(b) \otimes \cO_X(D)$ for some semireduced divisor $D$ of degree $g_X + 1 - (b - a)$ on $X$.
\end{proposition}

\begin{proof}
First suppose that $L$ has splitting type $(a,b)$. From the projection formula, we obtain
\[
  \tau_* (L \otimes \tau^* \cO(-b)) \simeq \tau_* L \otimes \cO(-b) \simeq \cO(a - b) \oplus \cO.
\]
Taking global sections yields
\[
    h^0(X, L \otimes \tau^* \cO(-b)) = h^0(\PP^1,  \cO(a - b) \oplus \cO) \geq 1. 
\]
Therefore $L \otimes \tau^* \cO(-b) = \cO_X(D)$ for some effective divisor $D$. Now note that
\[
    h^0(X,\cO_X(D-H)) = h^0(\PP^1, (\cO(a - b) \oplus \cO)(-1)) = 0,
\]
so $D$ contains no fibers of $\tau$. Finally we have
\[
  \deg(D) = \deg(L) - \tau^* \cO(b) = a+b + g_X + 1 - 2b = g_X + 1 - (b - a).
\]
as required.

Conversely suppose that $a \leq b$ are integers and $D$ is a semireduced divisor on $X$ of degree $g_X + 1 - (b-a)$. Let $(a',b')$ be the splitting type of $\cO_X(D)$. Because $D$ is effective, one obtains
\[
    \min\{0,a'+1\} + \min\{0,b'+1\} = h^0(\PP^1, \tau_* \cO_X(D)) = h^0(X,\cO_X(D)) \geq 1
\]
which implies $b' \geq 0$. Moreover, the push-pull formula and \Cref{lem:coh-of-semireduced} yields:
\begin{align*}
\min\{0,a'\} + \min\{0,b'\} &= h^0(\PP^1, \cO(-1) \otimes \tau_* \cO_X(D))  \\
&= h^0(\PP^1,\tau_* \cO_X(D-H)) \\
&= h^0(X,\cO_X(D-H)) \\
&= 0 
\end{align*}
and hence $b' \leq 0$. Thus $b' = 0$. Finally, we have
\[
    a' + b' + g_X + 1 = \deg(D) = g_X + 1 - (b-a),
\]
which yields $a' = -(b-a)$. Therefore $\tau_* \cO_X(D) = \cO(-(b-a)) \oplus \cO$ and another use of the push-pull formula yields
\[
    \tau_*(\tau^* \cO(b) \otimes \cO_X(D)) = \cO(b) \otimes \tau_* \cO_X(D) = \cO(a) \oplus \cO(b)
  \]
  as required.
\end{proof}

\subsection{Proof of \Cref{thm:double-triple}}

Let $\vec{e} \in \mathcal{Q}$ and let $\tau \colon X \rightarrow \PP^1$ be a hyperelliptic curve of genus $g_X = e_1 - 1$. Let $\{i,j\} \sqcup \{k,\ell\}$ be an admissible partition of $\vec{e}$. Let $D_1$ and $D_2$ be semireduced divisors (which will be chosen later) on $X$ with degrees $\deg(D_1) = g_X + 1 - (e_j - e_i)$ and $\deg(D_2) = g_X + 1 - (e_\ell - e_k)$ respectively. Consider the line bundles 
\[
  L_1 \colonequals \tau^* \cO(-e_i) \otimes \cO_X(D_1) \quad \text{and} \quad 
    L_2 \colonequals \tau^* \cO(-e_k) \otimes \cO_X(D_2).
\]
By \Cref{prop:splitting-types-of-hyp} the line bundles $L_1$ and $L_2$ have splitting type $(-e_j, -e_i)$ and $(-e_\ell, -e_k)$ respectively. In particular
\[
    L_2 \otimes L_1^{\otimes -2} = \tau^*\cO(2e_i - e_k) \otimes \cO_X(D_2 - 2D_1) = \cO_X((2e_i - e_k)H + D_2 - 2D_1).
\]
We claim that $D_1$ and $D_2$ can be chosen such that $h^0(X,L_2 \otimes L_1^{\otimes -2})$ has a nontrivial global section, i.e., that $(2e_i - e_k)H + D_2 - 2D_1$ is linearly equivalent to an effective divisor. 

Let $m = \min\{\deg(D_1), \lfloor \frac{1}{2}\deg(D_2)\rfloor \}$. Let $Q$ be a non-Weierstrass point on $X$ and choose $D_1$ and $D_2$ so that 
\[
    D_1 = mQ + D_1' \quad \text{and} \quad D_2 = 2mQ + D_2'    
\]
for effective divisors $D_1'$ and $D_2'$. Further, choose $D_1'$ to be a sum of distinct Weierstrass points, and choose $D_2'$ to be a semireduced divisor not containing $Q$. These choices ensure that $D_1$ and $D_2$ are semireduced. 

Now, because $2D_1'$ is linearly equivalent to $H$, the divisor $(2e_i - e_k)H + D_2 - 2D_1$ is linearly equivalent to
\begin{equation}
\label{eqn:divisor}
  \Delta = (2e_i - e_k - \deg(D_1'))H + D_2'
\end{equation}
The condition that $\{i,j\} \sqcup \{k,\ell\}$ is admissible implies that $\deg(\Delta) \geq 0$. We claim $\Delta$ is effective.

First suppose $D_1' = 0$. The condition that $\vec{e} \in \mathcal{Q}$ implies that $2e_i - e_k \geq 0$. Hence $\Delta$ is effective because $D_2'$ is effective. On the other hand if $D_1' \neq 0$, then $0 \leq \deg(D_2') \leq 1$ and
\[
    \deg(\Delta) = 2(2e_i - e_k - \deg(D_1')) + 1 \geq 0.
\]
In particular $2e_i - e_k - \deg(D_1') \geq 0$ and $\Delta$ is effective. 

By \Cref{prop:main-triple-cover-of-double} we conclude that there exists a triple cover $\gamma \colon C \rightarrow X$ with $\gamma_* \cO_C = \cO_X \oplus L_1 \oplus L_2$. The composition $\tau \circ \gamma \colon C \rightarrow \PP^1$ is a sextic cover with scrollar invariants $\vec{e}$. \hfill \qed

\appendix
\section{Tschirnhausen bundles of primitive degree \texorpdfstring{$6$}{6} covers}
\label{sec:primitive}
In this section, we sketch a potential application of \Cref{thm:main-thm} towards the following conjecture in the case $d = 6$.

\begin{conjecture}({\cite[Conjecture~1.3]{VV}})
A $(d-1)$-tuple $\vec{e}$ arises as the scrollar invariants of a smooth primitive degree $d$ cover if and only if $e_{i+j} \leq e_i + e_j$ for all $i,j$. 
\end{conjecture}

As discussed in the introduction, the case $d = 6$ is the smallest degree in which this conjecture is not known. In \cite[Theorem~1.4]{VV} it is shown that if $\vec{e}_\pi$ is the scrollar invariants of a primitive cover, then $e_{i + j} \leq e_i + e_j$ for all $i$ and $j$. It therefore remains to produce smooth \emph{primitive} covers with prescribed scrollar invariants. 

Define
\[
  \mathcal{P}_6 \colonequals \left\{
    (e_1,\dots,e_5) \in \ZZ^5 \quad \Bigg| \quad 
    \begin{aligned}
      & 1 \leq e_1 \leq \dots \leq e_5,  \\
      & e_{i+j} \leq e_i + e_j \text{    for all     } i,j 
    \end{aligned}
  \right\}.
\]
Let $\mathcal{H}_{6,g}$ be the Hurwitz stack of degree $6$ covers of $\PP^1_{\CC}$ by smooth irreducible genus $g$ curves. Let $\mathcal{M}(\vec{e}) = \{\pi \in \mathcal{H}_{6,g} : \vec{e}_{\pi} = \vec{e} \}$ and let $\mathcal{M}^{\fac}(\vec{e}) \subseteq \mathcal{M}(\vec{e})$ be the locus of covers which factor through a nontrivial proper subcover. Let $\codim \mathcal{M}(\vec{e})$ and $\codim \mathcal{M}^{\fac}(\vec{e})$ be the respective codimensions of $\mathcal{M}(\vec{e})$ and $\mathcal{M}^{\fac}(\vec{e})$ in $\mathcal{H}_{6,g}$.

\begin{proposition}
If $\codim \mathcal{M}^{\fac}(\vec{e}) > \sum_{i < j} \max\{0, e_j - e_i - 1\}$ then there exists a smooth primitive cover with scrollar invariants $\vec{e}$.
\end{proposition}
\begin{proof}
Consider the Tschirnhausen morphism $\Tsch \colon \mathcal{H}_{6,g} \rightarrow \Bun_{\PP^1}$  sending a cover $\pi \colon C \rightarrow \PP^1$ to the Tschirnhausen bundle $\mathcal{E}_\pi \colonequals (\pi_* \cO_C/\cO_{\PP^1})^{\vee}$. \Cref{thm:main-thm} implies that if $\vec{e} \in \mathcal{P}_6$ then $\mathcal{M}(\vec{e}) \neq \emptyset$. The upper semi-continuity theorem now implies that if $\vec{e} \in \mathcal{P}_6$, then
\[
    \codim \mathcal{M}(\vec{e}) \leq h^1(\PP^1, \End(\cO(e_1) \oplus \dots \oplus \cO(e_{5}))) = \sum_{i < j} \max\{0, e_j - e_i - 1\}.
\]
By assumption we have 
\[
    \codim \mathcal{M}^{\fac}(\vec{e}) > \sum_{i < j} \max\{0, e_j - e_i - 1\} \geq \codim \mathcal{M}(\vec{e})
\]
so therefore $\mathcal{M}(\vec{e})$ contains a smooth primitive cover. 
\end{proof}

\begingroup
\hypersetup{
  urlcolor=mybiburlcolor
}
\newcommand{\etalchar}[1]{$^{#1}$}
\providecommand{\bysame}{\leavevmode\hbox to3em{\hrulefill}\thinspace}
\providecommand{\MR}{\relax\ifhmode\unskip\space\fi MR }
\providecommand{\MRhref}[2]{%
  \href{http://www.ams.org/mathscinet-getitem?mr=#1}{#2}
}
\providecommand{\bibtitleref}[2]{%
  \hypersetup{urlbordercolor=0.8 1 1}%
  \href{#1}{#2}%
  \hypersetup{urlbordercolor=cyan}%
}
\providecommand{\href}[2]{#2}

\endgroup

\end{document}